\documentclass{elsarticle} 

\usepackage{hyperref} 

\usepackage{amsmath,amssymb}
\usepackage{amsthm}

\usepackage{xcolor}

\newtheorem{theorem}{Theorem}

\newtheorem{lemma}{Lemma}
\newtheorem{definition}{Definition}
\newtheorem{remark}{Remark}
\newtheorem{example}{Example}
\newtheorem{proposition}{Proposition}

\newcommand{\cM}{{\cal{M}}}
\newcommand{\re}{{\mathbb R}}
\newcommand{\ignore}[1]{{}}

\newcommand{\n}{{\mathbb N}}

\begin{document}

\begin{frontmatter}

\title{{Lower bounds and dense discontinuity phenomena for the stabilizability radius of linear switched systems}}

\author[address1]{Carl P. Dettmann}
\ead{carl.dettmann@bristol.ac.uk}

\author[address2]{R. M. Jungers}
\ead{raphael.jungers@uclouvain.be}

\author[address3]{P. Mason}
\ead{paolo.mason@centralesupelec.fr}

\address[address1]{School of Mathematics, University of Bristol, Fry Building, Woodland Road, Bristol BS81UG, UK.}
\address[address2]{ICTEAM Institute, Universit\'e catholique de Louvain, 4 avenue Georges Lemaitre, B-1348 Louvain-la-Neuve, Belgium. R. J. is an F.R.S.-FNRS honorary research associate.}
\address[address3]{Universit\'e Paris-Saclay, CNRS, CentraleSup\'elec, Laboratoire des signaux et syst\`emes, 91190, Gif-sur-Yvette, France.}

\begin{abstract}
We investigate the stabilizability of discrete-time linear switched systems, when the sole control action of the controller is the switching signal, and when the controller has access to the state of the system in real time. 
{Despite their apparent simplicity, determining if such systems are stabilizable appears to be a very challenging problem, and  basic examples have been known for long, for which the stabilizability question is open.}

We provide {new results} allowing us to bound the so-called stabilizability radius, which characterizes the stabilizability property of discrete-time linear switched systems.  
{These results allow us to compute significantly improved explicit lower bounds on
the stabilizability radius for the above-mentioned examples.}  As a by-product, we exhibit a discontinuity property for this problem, which brings theoretical understanding of its complexity.
\end{abstract}


\end{frontmatter}


\section{Introduction}

\emph{Joint spectral characteristics} are numerical quantities that describe the asymptotic behaviour of matrix semigroups.  They have found many applications, in particular in Systems and Control.  

Consider a finite set of $m$ matrices $\cM \in \re^{n\times n},$ and the corresponding \emph{linear discrete time switching systems,} which is a system whose behaviour follows the following law:
 \begin{equation} \label{ss} x(k+1) = A_{\sigma_k}x(k)\quad \sigma_k\in \{1,\dots, m\}. \end{equation}
 These systems are not uniquely defined, but any `switching signal' $\sigma$ implies a well defined law of evolution for the system.  The joint spectral characteristics have emerged quite independently during the second half of the 20th century, with the goal of characterizing the rate of growth of System \eqref{ss} for some possible switching signal.  These quantities have attracted a lot of attention, not only because of their applications, but probably also because, despite the apparent simplicity of their definition, they turn out to be extremely hard to compute.  See for instance \cite{BlTi3,BlTi2} for typical complexity results on the topic.

The first quantity, and perhaps the most well-known, was introduced in the context of robust control, and represents the worst case rate of growth of a switching system:
$$ \rho_\infty(\cM)=\lim_{k\rightarrow \infty}\max_{A\in \cM^k}\{||A||^{1/k}\}.$$  It is commonly referred to as the \emph{Joint Spectral Radius} (JSR in short) of the set $\cM.$ It has been introduced by Rota and Strang \cite{rs60}.  See \cite{jungers_lncis} for a monograph on the topic.
Since then, several other quantities were proposed, in order to describe other possible rates of growth of the system. Let us mention the \emph{p-radius,} \cite{Jia,Wang} with motivations in mathematical analysis; the \emph{Lyapunov exponent} (see \cite{Protasov-jungers-lyap-laa, pollicott}) with motivations in randomly switching systems; or the \emph{Joint Spectral Subradius}, which represents the minimal 
rate of growth for {the evolution operator of} System \eqref{ss} (see \cite{gu1,GP11}).  

{In this paper, we are concerned with the \emph{stabilizability radius}, which, similarly to the subradius, is also 
 related to the smallest possible rate of growth over all switching signals, but now it is assumed that one can choose the matrix sequence \emph{depending on the initial condition $x(0).$}} 
The stabilizability radius is thus smaller than the previously introduced subradius.  It has only been introduced formally recently \cite{SICON-feedback}, but the reader can find earlier implicit studies of it in \cite{geromel2006stability,geromel2006stabilityd,stanford}. 

Following~\cite{SICON-feedback}, we introduce the following definition,
which is the main topic of study of the present work:
\begin{definition}
The 
{stabilizability} radius of $\mathcal{M}$ is defined as
\[\tilde{\rho}(\mathcal{M}) = \sup_{x_0\in\mathbb{R}^n}\tilde{\rho}_{x_0}(\mathcal{M}),\]
where
\[\tilde{\rho}_{x_0}(\mathcal{M})\triangleq \inf\left\{\lambda\geq 0 \ \Bigl| \ 
\begin{array}{c}
\exists x(\cdot) \mbox{ solution of \eqref{ss} with initial point }x_0
\mbox{ and }\\ M>0 \mbox{ s.t. } |x(k)|\leq M \lambda^k |x_0| \mbox{ for any }
k\geq 0
\end{array}
\right\}\]
minimizes the exponential growth rate of the trajectories of~\eqref{ss} starting at $x_0$.
As shown in~\cite{SICON-feedback} we may equivalently write
\[\tilde{\rho}(\mathcal{M}) = \inf\left\{\lambda\geq 0 \ \Bigl| \ 
\begin{array}{c}
\exists M>0 \mbox{ s.t. } |x(k)|\leq M \lambda^k |x_0| \mbox{ for any }x_0\in\mathbb{R}^n,\\
k\geq 0\mbox{ and for some solution $x(\cdot)$ of~\eqref{ss} starting at }x_0
\end{array}
\right\}.\]
\end{definition}

The 
stabilizability radius is related to the possibility of stabilizing~\eqref{ss} by appropriately choosing the switching law, either in open loop form or in feedback form, see~\cite[Proposition~2.5 and Corollary~3.4]{SICON-feedback}. 
{ Note that the systems of the form~\eqref{ss} represent a special class of the systems considered in nonlinear control~\cite{sontag-book} and of variable structure systems~\cite{yan2017variable}, where stabilization issues play a crucial role. Hence the study of the stabilizability radius is an important step to understand the complexity of the stabilization problem in a more general context than the one considered in this paper.}

We recall the following basic properties {of the stabilizability radius}:
\begin{proposition}
\label{lem-basic}
The 
{stabilizability}  radius satisfies  the following basic properties:
\begin{itemize}
\item[(i)] Homogeneity: For any compact set of matrices $\cM,$ $\forall \gamma >0,$ $\tilde \rho (\gamma \cM) = \gamma \tilde \rho (\cM)$,
\item[(ii)] For any compact set of matrices $\cM,$ $\forall k \in \n,$ $\tilde \rho (\cM^k) = \tilde \rho (\cM)^k,$ {where $\cM^k$ denotes the set of $k$ products of matrices in $\mathcal{M}$.}
\end{itemize}
\end{proposition}

We illustrate the above concept with an example, which we will use as a running example throughout this paper:
\begin{example}\label{ex-urbano}[Based on an example by Stanford and Urbano~\cite{stanford}]
Let us consider $\mathcal{M}=\{A_1,A_2\}$, where
\[A_1=\left(\begin{array}{cc}\cos\frac{\pi}4 & \sin\frac{\pi}4 \\ -\sin\frac{\pi}4 & \cos\frac{\pi}4 \end{array}\right)=\frac{\sqrt{2}}{2}\left(\begin{array}{cc}1&1\\-1&1\end{array}\right), \quad A_2=\left(\begin{array}{cc} \frac12 & 0 \\ 0 & 2  \end{array}\right).\]
It is easy to see that the norm of any product of matrices in $\cM$ is larger than or equal to one. Indeed, $\det (A_{\sigma(k)}\dots A_{\sigma(0)})=1$ independently on the switching sequence, which implies that $\|A_{\sigma(k)}\dots A_{\sigma(0)}\|\geq 1.$ 
However, the definition of the 
{stabilizability} radius allows the switching sequence {to depend on the value of $x(0),$ so that the 
{stabilizability} radius can be smaller than one, and it is the case in our example.  
Indeed, for any value of $x\in \mathbb{R}^2$} there always exists a natural number $n_x\leq 3$ such that the absolute value of the angle formed by the vector $A_1^{n_x}x$ and the $x_1$ axis is smaller than or equal to $\pi/8$. As a consequence it is easy to obtain the estimate $|A_2A_1^{n_x}x|<0.9 |x|$.
{Hence, starting from any initial condition $x(0)$ we can easily construct recursively a switching sequence in such a way that the corresponding solution $x(k)$ of System \eqref{ss} satisfies $|x(k)|\leq 2 \times 0.9^{k/4}|x(0)|$, which implies $\tilde\rho(\cM)<0.9^{1/4}\sim 0.974$.}
\end{example}

In~\cite[Theorem~4.7]{SICON-feedback} it was shown that the minimum singular value computed among all matrices of $\mathcal{M}$ provides a lower bound for $\tilde{\rho}(\mathcal{M}):$
\begin{theorem}\label{lem-sing} \cite[Theorem~4.7]{SICON-feedback}
One has $\tilde\rho(\cM)\geq \min_{A\in\mathcal{M}}\sigma_{m}(A)$ where $\sigma_{m}(A)$ is the smallest singular value of the matrix $A$.
\end{theorem}
 However, the study of Example \ref{ex-urbano} seems to suggest that such a lower bound does not represent a good approximation of the actual value of the 
 {stabilizability} radius:

\begin{example} (Example \ref{ex-urbano}, continued.) A simple application of Theorem~\ref{lem-sing} to the matrices from Example \ref{ex-urbano} gives $\tilde \rho(\cM) \geq 1/2$, and applying the same result to $\cM^k,$ together with Item~$(ii)$ in Proposition~\ref{lem-basic}, does not improve the bound. \end{example}

This raises the following open question \cite[Open Question 2]{SICON-feedback}: \textit{Is it possible to improve Theorem \ref{lem-sing} and provide a generally better formula for a lower bound? In particular, how can one compute a better lower bound on $\tilde{\rho}(\mathcal{M})$ in Example \ref{ex-urbano}?}

Some techniques have been proposed in the control literature, which allow to derive an upper bound on the 
{stabilizability}  radius, mainly based on semidefinite programming  \cite{fiacchini-jungers,geromel2006stability,fiacchini-girard-jungers}. However, it seems much harder to provide a tight lower bound.
In this manuscript we tackle the above question.  By a closer inspection at all the singular values of a matrix product, we provide a much better lower bound, which can be improved by increasing the length of the products. In particular, we improve the lower bound previously obtained by applying Theorem \ref{lem-sing} to Example~\ref{ex-urbano}.
Then, in Section \ref{sec-dependence}, we show that the present result is actually general and has rather nonintuitive consequences concerning the regularity of the radius $\tilde{\rho}_{x_0}$ in terms of the initial condition $x_0.$ 

{\textbf{Notation:} {
For a matrix $M\in \mathbb{R}^{n\times n}$ we denote by $s_1(M)\leq s_2(M)\dots \leq s_n(M)$ the corresponding singular values.
The sphere in $\mathbb{R}^n$ is $S^{n-1}=\{x\in\mathbb{R}^n \,\vert\, \|x\|=1\}$.
Finally, we denote by $\mathrm{clos}(A)$ the closure of a subset $A$ of a topological space.}

\section{Lower bound for the {stabilizability}  radius} 

In this section we first provide a simple and actionable lower bound on the 
{stabilizability}  radius, based on the determinants of the matrices in $\cM.$ 
{We then provide a  more powerful bound, which is in some sense less actionable because it relies on more involved computations.}
{Theorem~\ref{lem-sing} provides} a simple lower bound on the 
{stabilizability} radius 
in terms of the smallest singular value of the matrices. We start with a simple lemma that pushes this reasoning further, by 
{ pointing out a geometric property of a given matrix related both to its smallest singular value and to its determinant.}

\begin{lemma}
Let $n\geq 2$. Given a matrix $A\in \mathbb{R}^{n\times n}$ and $r>0$ we define $S_{r,A}=\{x\in S^{n-1}\,\vert\,\|Ax\|\leq r\}$. Then 
\begin{itemize}
\item $S_{r,A}$ is empty if $s_1(A)>r$,
\item for every nonsingular matrix $A\in \mathbb{R}^{n\times n}$, and if  $s_1(A)\leq r$, 
the measure of $S_{r,A}$  is bounded by $m_{n-1} \min\{r^n|\det A|^{-1},1\}$, 
where $m_{n-1}$ is the surface area of the unit sphere.
\end{itemize}
\label{measure}
 \end{lemma}
\begin{proof}
The fact that the set $S_{r,A}$ is empty if $s_1(A)>r$ follows from the definitions.
In order to prove the second part of the lemma
we define the cone
\[C_{r,A} = \{x\in\mathbb{R}^n\,\vert\,\|x\|\leq 1,\ x/\|x\|\in S_{r,A}\}\] 
and consider the image of $C_{r,A}$ by the matrix $A$.
We know that the corresponding volumes scale by a factor $|\det(A)|$ and that the image of $C_{r,A}$ is completely contained inside the closed ball of radius $r$.
We deduce that 
\[V_{r,A}\leq  M_n \min\{r^n|\det(A)|^{-1},1\}\]
where $V_{r,A}$ is the volume of $C_{r,A}$ and $M_n$ denotes the volume of the unit ball. The desired bound is then obtained by observing that the surface area of $S_{r,A}$ is equal to $V_{r,A}m_{n-1}/M_n$.
\end{proof}
\begin{remark}
It is possible to show that the surface area of $S_{r,A}$ is bounded by $2(\pi r)^{n-1}s_1(A)|\det A|^{-1}$ whenever $A$ is nonsingular and $s_1(A)\leq r$. 
This improves the estimate provided by Lemma~\ref{measure} in the case $s_1(A)\ll r\leq |\det(A)|^{1/n}$.
We omit the proof of this result as it is more involved than that of Lemma~\ref{measure} and the improved estimate does not allow to enhance the later results.
\end{remark}

Below we exploit the previous lemma in order to provide a first lower bound to $\tilde{\rho}(\mathcal{M})$. Roughly speaking, the idea is that for any $\lambda>\tilde{\rho}(\mathcal{M})$ and any $x\in S^{n-1}$ there should exist a matrix $A$ in $\mathcal{M}^k$, for $k$ large enough, such that $Ax$ belongs to the ball of radius $\lambda^k$. In other words, in the notations of the previous lemma, the union of all the sets $S_{\lambda^k,A}$ for $A\in \mathcal{M}^k$ must cover the whole sphere $S^{n-1}$.

\begin{theorem}
\label{lower-bound-0} Consider System \eqref{ss} and assume that $\cM$ only contains nonsingular matrices. Then, the stabilizing radius satisfies
\begin{equation}\label{eq-lower-bound}
\tilde{\rho}(\mathcal{M})\geq \tilde{\rho}_-\triangleq\Big(\sum_{h=1}^m|\det A_h|^{-1}\Big)^{-1/n}.
\end{equation}
\end{theorem}
\begin{proof}
Let us fix $\lambda = \rho + \epsilon$ for some $\epsilon>0,$ and fix some $T \in\mathbb N. $ By Lemma~\ref{measure}, for any product $A\in\mathcal{M}^T,$ the set $S_{r,A}$ of unit vectors which are mapped inside the ball of radius $r$ by $A$ has measure bounded by $m_{n-1} r^n |\det(A)|^{-1}$, where $m_{n-1}$ is the surface area of the unit sphere. This implies that the set $\cup_{A\in \mathcal{M}^T} S_{r,A}$ has measure bounded by 
\begin{align*}
m_{n-1} r^n \sum_{A\in \mathcal{M}^T}|\det(A)|^{-1}
 & = m_{n-1} r^n \sum_{\sigma\in \{1,\ldots,m\}^T}\prod_{i=1,\dots,T}|\det(A_{\sigma_i})|^{-1} \\
& = m_{n-1} r^n\Big(\sum_{h=1}^m|\det(A_h)|^{-1}\Big)^{T}.
\end{align*}
Now, by definition, for any $\lambda>\tilde{\rho}(\mathcal{M})$ there exists a positive constant $C$ such that any $x\in S^{n-1}$ may be mapped to a ball of radius $C\lambda^T$ by at least one product of $T$ matrices in $\mathcal{M}$. Setting $r=C\lambda^T$, we deduce that the set $\cup_{A\in\mathcal{M}^T} S_{r,A}$ must cover the whole sphere $S^{n-1}$, so that 
\[m_{n-1} C^n\lambda^{nT}\Big(\sum_{h=1}^m |\det(A_h)|^{-1}\Big)^{T}\geq m_{n-1}.\]
Letting $T$ tend to infinity we get 
\[\lambda\geq \Big(\sum_{h=1}^m|\det(A_h)|^{-1}\Big)^{-1/n},\]
and the thesis follows since, by definition, $\lambda=\tilde{\rho}+\epsilon$ for an arbitrary small $\epsilon.$
\end{proof}

\begin{example} (Example \ref{ex-urbano}, continued.) 
Despite the strikingly simple idea leading to it, the previous result allows us to provide an answer to~\cite[Open Question 2]{SICON-feedback}.
Indeed, let us consider again the set $\cM$ in Example \ref{ex-urbano}.
One has that $\tilde{\rho}_- = \sqrt{2}/2$, considerably improving the lower bound $\min\{s_1(A_1),s_1(A_2)\}=1/2$ obtained applying~\cite[Theorem~4.7]{SICON-feedback}. \end{example}

It is worth noticing that, for an arbitrary given set $\mathcal{M}$, Theorem~\ref{lower-bound-0} might not necessarily improve on the simple lower bound from Theorem~\ref{lem-sing}: $\tilde{\rho}(\mathcal{M})\geq \min_{A\in\mathcal{M}} s_1(A).$ (Indeed consider for instance a trivial example with two rotation matrices. For such an example we have $\min_A s_1(A) = 1$ and $\tilde{\rho}_- = \sqrt{2}/2.$) Intuitively, on the one hand 
{Theorem~\ref{lem-sing}} estimates the maximal norm contraction at each step, among all available matrices and all initial conditions, but it does not take into account the fact that trajectories may move away from the most contracting directions;
on the other hand, Theorem~\ref{lower-bound-0} is essentially based on the assumption of a homogeneous occupation measure, and it does not exploit the possible presence of privileged directions or modes which may be used for optimizing the contraction rate. The result below mingles the two approaches.

\begin{theorem}
\label{lower-bound}
 Consider System \eqref{ss} and assume that $\cM$ only contains nonsingular matrices. Let us denote, for simplicity, $\delta_i=s_{1}(A_i)$ and $\Delta_i=|\det A_i|$.
Consider the simplex $\Sigma_m$ defined as
\[\Sigma_m=\{\nu\in [0,1]^m\,\vert\,\sum_{h=1}^m\nu_h =1  \},\]
the map
\[
\Psi:\Sigma_m\to \mathbb{R},\qquad \Psi(\nu) =  \sum_{h=1}^m\nu_h \log\left(\frac{\nu_h\Delta_h}{\delta_h^n}\right)
\]
and the element $\bar \nu\in \Sigma_m$ whose components are defined by $\bar\nu_h=\frac{\Delta_h^{-1}}{\sum_{j=1}^m\Delta_j^{-1}}$.
Then we have the following alternative:
\begin{itemize}
\item[(a)] If $\Psi(\bar\nu)\geq 0$, then $\tilde{\rho}(\mathcal{M})\geq \tilde{\rho}_-\geq\min_{h=1,\dots,m}\delta_h$.
\item[(b)] If $\Psi(\bar\nu)<0$, then the set $\mathcal{Z}=\{\nu\in \Sigma_m\,\vert\,\Psi(\nu)=0\}$ is nonempty
and, setting  
\begin{equation}
\tilde{\rho}^*_-\triangleq \min_{\nu\in \mathcal{Z}}\prod_{h=1}^m\delta_h^{\nu_h}, 
\label{lagrange}
\end{equation}
we have $\tilde{\rho}^*_-\geq\min_{h=1,\dots,m}\delta_h$ and
\begin{equation}
\tilde{\rho}(\mathcal{M})\geq \tilde{\rho}^*_->\tilde{\rho}_-.
\label{eq-est-opt}
\end{equation}
Moreover, {if $\delta_i\neq \delta_j$ for every $i\neq j$,} the argument of the minimum in~\eqref{lagrange} takes the form 
\begin{equation}
\label{minimizing}
\hat{\nu}_h(\beta)=\left\{\begin{array}{ll} 0 & \mbox{if }h\notin S \\ 
\frac{ \delta_h^{\beta} \Delta_h^{-1}}{\sum_{j\in S} \delta_j^{\beta} \Delta_j^{-1}} & \mbox{if }h\in S\end{array}\right.
\end{equation} 
for some real value $\beta$ and $S\subseteq \{1,\dots,m\}$. As a consequence, $\tilde{\rho}^*_-$ may be calculated numerically by solving the scalar equations 
$\Psi(\hat{\nu}_1(\beta),\dots,\hat{\nu}_m(\beta))=0$ obtained for all possible $S\subset \{1,\dots,m\}$.
\end{itemize}
\end{theorem}

\begin{proof}
If  $A$ is a product of length $T$ containing $n_h$ copies of $A_h$  then $|\det A| = \prod_{h=1}^m\Delta_h^{n_h}$. By Lemma~\ref{measure}, the portion of the unit sphere that is mapped into a ball of radius $r$ has measure  bounded by $m_{n-1}r^n|\det A|^{-1}=m_{n-1}r^n \prod_{h=1}^m\Delta_h^{-n_h}$. 
Moreover one has that $\sum_{h=1}^m n_h\log\delta_h=\log(\prod_{h=1}^m \delta_h^{n_h})\leq\log(s_1(A))$ so that, applying the first item in Lemma~\ref{measure}, it follows that $S_{r,A}$ is empty whenever $\sum_{h=1}^m n_h\log\delta_h>\log r$.
Thus, we obtain the following upper bound on the measure of the union of all sets $S_{r,A}$ among all possible products of length $T$ of matrices of $\mathcal{M}$
\begin{equation}
U(T)=m_{n-1}r^n\sum\frac{T!}{\prod_h n_h!}\frac{1}{\prod_h\Delta_h^{n_h}}=m_{n-1}r^n\sum\frac{T!}{\prod_h (n_h!\Delta_h^{n_h})},
\label{total-area}
\end{equation}
where in the last two equalities the sum is taken over all $m$-tuples of positive integers satisfying 
\begin{equation}
\sum_{h=1}^m n_h\log\delta_h\leq\log r,\qquad \sum_{h=1}^m n_h=T.
\label{constraint2}
\end{equation}
Noticing that $\prod_h (n_h!\Delta_h^{n_h})=\prod_{h,n_h\neq 0} (n_h!\Delta_h^{n_h}),$ we can apply Stirling approximation
\[N!\approx N^{N+\frac12} e^{-N},\quad N>0\]
to~\eqref{total-area} obtaining that $c_1 \tilde{U}(T) \leq U(T) \leq c_2 \tilde{U}(T)$ for some positive numbers $c_1,c_2$ only depending on $m,n$, where
\begin{align*}
 \tilde{U}(T)&=r^n\sum\left(\frac{T}{\prod_{h,n_h\neq 0} n_h}\right)^{\frac12}\frac{T^T}{\prod_{h,n_h\neq 0} \left(n_h\Delta_h\right)^{n_h}}\\
&= r^n\sum\left(\frac{T}{\prod_{h,n_h\neq 0} n_h}\right)^{\frac12} \prod_{h,n_h\neq 0} \left(\frac{n_h}{T}\Delta_h\right)^{-n_h}.
\end{align*}
Since, for any $T>0$  $\frac{T}{\prod_{h,n_h\neq 0} n_h}\leq \frac{T}{\max_h n_h}\leq m$, the expression on the right is bounded by
\[\sqrt{m} r^n\sum\prod_h \left(\frac{n_h}{T}\Delta_h\right)^{-n_h}\leq \sqrt{m}r^n (T+1)^{m-1} \max \prod_h\left(\frac{n_h}{T}\Delta_h\right)^{-n_h}\]
where we estimate  the number of elements in the summation by $\binom{T+m-1}{m-1} \leq (T+1)^{m-1}$, and the maximum is taken over all $m$-tuples of positive integers satisfying~\eqref{constraint2}.

In particular, by replacing each $\frac{n_h}{T}$ with a continuous variable $\nu_h$, the subset of the unit sphere which can be mapped into the ball of radius $\hat C\rho^T$ has measure bounded by 
\begin{equation}
E(\rho,\hat C,T)\triangleq C\hat C^n\rho^{nT}(T+1)^{m-1} \left(\max_{\nu} \prod_h (\nu_h\Delta_h)^{-\nu_h}\right)^T
\label{estimate}
\end{equation}
for some $C>0$, where the $\nu_h$'s satisfy
\begin{subequations}
\begin{align}
&\sum_{h=1}^m \nu_h=1,~~\mbox{where}~~ \nu_h\geq 0~~\forall h,\label{constraint3b}\\
&\sum_{h=1}^m \nu_h\log\delta_h\leq\log \rho.\label{constraint3a}
\end{align}
\end{subequations}
Note that the constraint~\eqref{constraint3b} corresponds to $\nu\in \Sigma_m$.
Whenever $\rho>\tilde{\rho}(\mathcal{M})$ and for some $\hat C$ large enough, $E(\rho,\hat C,T)$ must necessarily be larger than the measure of the $(n-1)$-dimensional unit sphere for every integer $T>0$. In particular, setting
\[u(\rho)\triangleq \lim_{T\to\infty}\frac1T \log E(\rho,\hat C,T)=n\log \rho - \min_{\nu}\sum_h \nu_h\log  (\nu_h\Delta_h),\]
$\rho>\tilde{\rho}(\mathcal{M})$ implies $u(\rho)\geq 0$.
Since $u$ is strictly increasing, we actually have that $\rho>\tilde{\rho}(\mathcal{M})$ implies $u(\rho)> 0$.
As a consequence, if $u(\rho)\leq 0$ then $\rho\leq \tilde{\rho}(\mathcal{M})$ and the problem of finding the maximum $\rho$ satisfying $u(\rho)\leq 0$ under the constraints~\eqref{constraint3b}-\eqref{constraint3a} allows to determine a lower bound to  $\tilde{\rho}(\mathcal{M})$.
Note that $u(\rho)\leq 0$ and~\eqref{constraint3a} imply that the values of $\nu_h$ among which such a maximum must be seek  satisfy
\[
\sum_{h=1}^m \nu_h\log\delta_h\leq\log\rho\leq\frac1n\sum_{h=1}^m \nu_h\log  (\nu_h\Delta_h),
\]
and, as a consequence,
\begin{equation}
\Psi(\nu)=  \sum_{h=1}^m\nu_h \log(\nu_h\Delta_h)-n\sum_{h=1}^m \nu_h\log\delta_h\geq 0.
\label{constraint4}
\end{equation}
Note that~\eqref{constraint4} is always satisfied at the vertices of the simplex $\Sigma_m$, with equality if the corresponding matrix in $\mathcal{M}$ is proportional to an  orthogonal matrix, that is, $\Delta_h=\delta_h^n$.

From what precedes a lower bound $\rho_*$ for $\tilde{\rho}(\mathcal{M})$ should satisfy the following minimization problem:
\begin{equation}
\mbox{find }~\rho_*\triangleq\min_{\nu} e^{\frac1n \Phi(\nu)}\quad \mbox{subject to~\eqref{constraint3b}-\eqref{constraint4}}
\label{problem}
\end{equation}
where
\[\Phi(\nu)\triangleq \sum_{h=1}^m \nu_h\log  (\nu_h\Delta_h). \]
In particular $\Phi$ is convex, as it is the sum of convex functions of a single real variable. Also, $\Phi(\nu) = \Psi(\nu)+n\sum_{h=1}^m\nu_h\log\delta_h\geq n\sum_{h=1}^m\nu_h\log\delta_h\geq n\min_{h=1,\dots,m}\log\delta_h$, which implies 
\[\rho_*\geq \min_{h=1,\dots,m}\delta_h.\]

Consider now the problem of minimizing $\Phi$ under the sole condition~\eqref{constraint3b}, i.e. on the simplex $\Sigma_m$. Since $\frac{\partial\Phi}{\partial\nu_h}$ tends to $-\infty$ if $\nu_h$ goes to $0$, the minimum of $\Phi$ is not attained at the boundary of $\Sigma_m$ and can therefore be computed using Lagrange multipliers. In particular the value $\bar\nu\in\Sigma_m$ minimizing $\Phi$ is given by
\[\bar\nu_h=\frac{\Delta_h^{-1}}{\sum_j\Delta_j^{-1}},\]
and it is easy to see that $e^{\frac1n \Phi(\bar\nu)}=\tilde{\rho}_-$. Therefore, if $\Psi(\bar\nu)\geq 0$  we obtain that $\rho_*=\tilde{\rho}_-$ in~\eqref{problem}, concluding the proof of Item~(a). 

Assume now that $\Psi(\bar\nu)< 0$. We claim that the minimum in the definition of $\rho_*$ is attained when $\Psi$ is equal to $0$. 
Indeed, by continuity of $\Psi$, for any $\nu$ satisfying~\eqref{constraint4} there exists a convex combination $\nu_{\lambda}=\lambda\nu+(1-\lambda)\bar\nu$ such that $\Psi(\nu_{\lambda})=0$. Moreover $\Phi(\nu_{\lambda})\leq \Phi(\nu)$ (with equality only if $\nu=\nu_{\lambda}$) by convexity of $\Phi$ and since $\Phi(\bar \nu)<\Phi(\nu)$. 
Thus, without loss of generality, in the problem~\eqref{problem} one may replace the constraint~\eqref{constraint4} with $\Psi=0$, that is we can minimize $\Phi$ restricted to the subset $\mathcal{Z}$.
We observe that $\Phi(\nu)=\Psi(\nu)+n\sum_{h=1}^m\nu_h\log\delta_h=n\sum_{h=1}^m\nu_h\log\delta_h$ for $\nu\in \mathcal{Z}$, from which we deduce that $\rho_*=\tilde{\rho}^*_- = \min_{\nu\in \mathcal{Z}}\prod_{h=1}^m\delta_h^{\nu_h}$.  This proves the first inequality in~\eqref{eq-est-opt}. The last strict inequality in~\eqref{eq-est-opt} is a consequence of the uniqueness of the minimizer $\bar \nu$ obtained with the sole constraint~\eqref{constraint3b}.  

Concerning the last part of the theorem, if the minimum of $\Phi$ restricted to $\mathcal{Z}$ is attained in the interior of $\Sigma_m$ then it can be computed by using Lagrange multipliers. In particular, {if the values $\delta_i$ are all different,} one  finds that
\[\nu_h=\alpha \delta_h^{\beta} \Delta_h^{-1},\]
where $\alpha,\beta$ depend on the parameters $\delta_i,\Delta_i$, $i=1,\dots,m$.
Since $\nu\in \Sigma_m$, we get $\alpha=\alpha(\beta)=(\sum_h \delta_h^{\beta} \Delta_h^{-1})^{-1}$, so that, setting $\hat{\nu}_h(\beta)=\alpha(\beta)\delta_h^{\beta} \Delta_h^{-1}$ the value $\beta$ may be found numerically by solving the equation \[\Psi(\hat{\nu}(\beta))=0.\]
If the minimum of $\Phi$ restricted to $\mathcal{Z}$ is attained at the boundary of $\Sigma_m$ then either it is attained at one vertex of $\Sigma_m$, or in the interior of a subsimplex $\{\nu\in \Sigma_m\,\vert\,\nu_h =0,\ \forall h\notin S\}$, for some $S\subset \{1,\dots,m\}$. In the latter case it minimizes the restriction of $\Phi$ to that subsimplex under the constraint $\Psi=0$ and one can again find the minimizer by means of Lagrange multipliers, obtaining that $\min_{\nu\in \mathcal{Z}}\Phi(\nu)$ is attained at a point $\hat \nu$ of the form~\eqref{minimizing}.
This concludes the proof of Item~(b). 
 \end{proof}

\begin{remark}
\label{rem-improve}
Theorem~\ref{lower-bound} shows that the inequality~$\tilde{\rho}(\mathcal{M})\geq \min_{h=1,\dots,m}\delta_h$ (first provided in~\cite[Theorem~4.7]{SICON-feedback}) is actually strict, except for very special cases.
In particular the strict inequality holds if $\min_{h=1,\dots,m} \delta_h$ is attained only for a single index $h=\bar h$ and $A_{\bar h}$ is not proportional to an orthogonal matrix (that is, if $\Delta_{\bar h}>\delta_{\bar h}^n$). { On the other hand, the lower bound obtained in Theorem~\ref{lower-bound} in the case $(a)$ coincides with the one obtained in Theorem~\ref{lower-bound-0}.}
\end{remark}

\begin{remark}
There are at least two simple ways to possibly improve the lower bound in Theorem~\ref{lower-bound}:
\begin{itemize}
\item Unlike $\tilde{\rho}(\mathcal{M})$, the value $\tilde{\rho}^*_-$ in Theorem~\ref{lower-bound} may actually vary if one performs a  linear coordinate transformation (common to each $A\in \mathcal{M}$), as the singular values are not invariant with respect to linear coordinate transformations. Therefore one can consider the problem of optimizing the lower bound by coordinate changes.
\item We have $\tilde{\rho}(\mathcal{M})=\tilde{\rho}(\mathcal{M}^k)^{1/k}$ (see Proposition~\ref{lem-basic}). In particular, computing a lower bound for  $\tilde{\rho}(\mathcal{M}^k)$ for $k>1$ by means of Theorem~\ref{lower-bound} may lead to better estimate of $\tilde{\rho}(\mathcal{M})$ compared to a direct application of Theorem~\ref{lower-bound} to the set $\mathcal{M}$.
\end{itemize}
We show below through a simple example that Theorem \ref{lower-bound} may strictly increase the lower bound on $\tilde{\rho}(\mathcal{M})$. Whether the iteration of such a method leads asymptotically to the actual value $\tilde{\rho}(\mathcal{M})$ remains an open problem.  
\label{rem-improve-2}
\end{remark}

 \begin{example}
 \label{Stanford-revisited}
 To illustrate the previous result we consider a set of matrices $\mathcal{M}=\{A_1,A_2\}$ where $A_1=\mathrm{diag}(c,c^{-1})$ with $c\in(0,1)$, and $A_2$ is a two-by-two orthogonal matrix. In particular the matrices in {Example~\ref{ex-urbano}} 
 satisfy such assumptions with $c=1/2$. In the notation of Theorem~\ref{lower-bound} we have $\delta_1=c,\ \delta_2=1$ and $\Delta_1=\Delta_2=1$. Moreover, $\bar\nu_1=\bar\nu_2=1/2$ and $\Psi(\bar\nu) = \log \frac12 -\log c$. Thus, for $c\in (0,1/2]$ we fall into case $(a)$ of the theorem; the lower bound $\tilde{\rho}_-=1/\sqrt{2}$ provided by both Theorem~\ref{lower-bound-0} and Theorem~\ref{lower-bound}  improves the value $\delta_{\rm min}\triangleq\min\{\delta_1,\delta_2\}=c$ of \cite[Theorem~4.7]{SICON-feedback}. On the other hand, if $c\in (1/2,1)$ we fall into case $(b)$ of Theorem~\ref{lower-bound}, and the lower bound $\tilde{\rho}^*_-$ is strictly larger than both $\tilde{\rho}_-$ and $\delta_{\rm min}$.
 In this case it is easy to see that the minimum is attained at the interior of $\Sigma_2$ and that it is associated with the unique solution of  the equation $\Psi(\hat{\nu}_1(\beta),\dots,\hat{\nu}_m(\beta))=0$, obtained for $S = \{1,2\}$, which may be easily found numerically. In Table~\ref{table-example} we collect the lower bounds $\delta_{\rm min},\tilde{\rho}_-,\tilde{\rho}^*_-$ for different values of $c\in (0,1)$.
  \end{example}

 \begin{table}
 \begin{center}
\begin{tabular}{|c|c|c|c|c|c|c|c|c|}
\hline
 c & 0.2 & 0.3 & 0.4 & 0.5 & 0.6 & 0.7 & 0.8 & 0.9\\
\hline
 $\delta_{\rm min}$ & 0.2 & 0.3 & 0.4 & 0.5 & 0.6 & 0.7 & 0.8 & 0.9\\
$\tilde{\rho}_-$ & {\it 0.7071} & {\it 0.7071}  & {\it 0.7071} & {\it 0.7071} & 0.7071 & 0.7071 & 0.7071 & 0.7071\\
$\tilde{\rho}^*_-$ & - & - & - & {\it0.7071} & {\it 0.7212} & {\it 0.7613} & {\it 0.8236} & {\it 0.9048}\\
\hline
\end{tabular}
   \caption{\label{table-example}Computation of   $\delta_{\rm min},\tilde{\rho}_-$ and $\tilde{\rho}_-^*$ in terms of the parameter $c$ appearing in  the set $\mathcal{M}$ of Example~\ref{Stanford-revisited}. 
   Italic numbers represent the best bound.}
\end{center}
\end{table}

\section{Dependence of the 
{stabilizability} radius on the initial condition} \label{sec-dependence}

We consider now an application of Theorem~\ref{lower-bound}. We are interested in studying the dependence of $\tilde{\rho}_{x_0}(\mathcal{M})$ on the initial condition $x_0$. In general, one cannot expect this function  to be everywhere continuous. For instance, if $\mathcal{M}$ is made of a single matrix $A=\mathrm{diag}\{\lambda_1,\dots,\lambda_n\}$, then the image of 
{such a function} is equal to $\{|\lambda_1|,\dots,|\lambda_{n}|\}$ and $\tilde{\rho}_{x_0}(\mathcal{M})=|\lambda_i|$ if $(x_0)_i\neq 0$ and $(x_0)_j = 0$ for all $j$ such that  $|\lambda_i|\leq |\lambda_j|$. In particular $\tilde{\rho}_{x_0}(\mathcal{M})$ is discontinuous at any point $x_0$ with zero component along the eigenspaces corresponding to the eigenvalues of  maximum absolute value.  We show below a much more surprising result, which entails the existence of linear switched systems such that $\tilde{\rho}_{x_0}(\mathcal{M})$ is nowhere continuous.
Namely, we provide general conditions on the linear switched system ensuring the existence of an open set $\mathcal{A}\subset\mathbb{R}^n$ and two disjoint subsets $U,S$ both dense in $\mathcal{A}$, such that $\tilde{\rho}_{x_0}(\mathcal{M})\geq c_1$ for $x_0\in U$ and $\tilde{\rho}_{x_0}(\mathcal{M})\leq c_2$ for $x_0\in S$ for some $c_1>c_2>0$.\\
In other words, up to multiplying the matrices of $\mathcal{M}$ by a common constant, from each point of $S$ it is possible to stabilize exponentially the system and from each point of $U$ it is impossible to stabilize the system.

We start with the following definitions which adapt classical notions from continuous-time controlled dynamical systems. 
\begin{definition}
Consider the discrete-time switched system on the projective space $\mathbb{P}^{n-1}$ (for simplicity, we identify here a point of $\mathbb{P}^{n-1}$ with a pair of opposite points $z,-z$ in $S^{n-1}$) 
\begin{equation}
z(k+1)=\frac{B_{\sigma_k}z(k)}{|B_{\sigma_k} z(k)|},\quad B_{\sigma_k}\in \mathcal{N}\subset \mathbb{R}^{n\times n},
\label{eq-proj}
\end{equation}
and denote the attainable set in positive time $O^+(z(0))$ from $z(0)$ as the set of all points that can be reached from $z(0)$  for all $k\geq 0$ and switching signals $\sigma$, that is $O^+(z(0))=\big\{\frac{Bz(0)}{|Bz(0)|}\,\vert\, B\in \mathcal{N}^k,\ k\geq 0\big\}$. If $\mathcal{N}$ is made by nonsingular matrices, we also define the attainable set in negative time from $z(0)$ as $O^-(z(0))=\big\{\frac{B^{-1}z(0)}{|B^{-1}z(0)|}\,\vert\, B\in \mathcal{N}^k,\ k\geq 0\big\}$.
\end{definition}

We have the following general result concerning the 
{stabilizability} rate $\tilde{\rho}_{x_0}(\mathcal{M})$.
\begin{proposition}
\label{discontinuous}
 Consider System \eqref{ss} and assume that  $\mathcal{M}=\{A_1,\dots,A_m\}\subset \mathbb{R}^{n\times n}$ only contains nonsingular matrices. 
Consider the projected switched system on $\mathbb{P}^{n-1}$
\begin{equation}
\label{eq-proj-1}
z(k+1)=\frac{A_{\sigma_k}z(k)}{|A_{\sigma_k} z(k)|},\quad A_{\sigma_k}\in \mathcal{M}.
\end{equation}
Assume that there exist two points $z^{(1)},z^{(2)}\in\mathbb{P}^{n-1}\subset \mathbb{R}^n$ such that $\tilde{\rho}_{z^{(1)}}(\mathcal{M})<\tilde{\rho}_{z^{(2)}}(\mathcal{M})$. Then the function $x\mapsto \tilde{\rho}_{x}(\mathcal{M})$ is discontinuous at every point of the cone $$D=\{x\in \mathbb{R}^n: x= \lambda \ (\mathrm{clos} (O^-(z^{(1)}))\cap \mathrm{clos} (O^+(z^{(2)}))),\ \lambda >0\}.$$
\end{proposition}
\begin{proof}
Clearly $\tilde{\rho}_{z}(\mathcal{M})\leq \tilde{\rho}_{z^{(1)}}(\mathcal{M})$ if $z\in O^-(z^{(1)})$ and $\tilde{\rho}_{z}(\mathcal{M})\geq \tilde{\rho}_{z^{(2)}}(\mathcal{M})$ if $z\in O^+(z^{(2)})$. The conclusion follows since  both $O^-(z^{(1)})$ and $O^+(z^{(2)})$ are dense in $D$.
\end{proof}
Of course the previous result is of some interest only if the set $D$ is nonempty. This is the case under some (approximate) controllability property of the system. For instance, if $n=2$ and if there exists $k\in\mathbb{N}$ and $A\in\mathcal{M}^k$ with nonreal eigenvalues such that $A^h$ is not proportional to the identity for every integer $h\neq 0$ then $\mathrm{clos} (O^+(z)) = \mathrm{clos} (O^-(z)) = \mathbb{P}^{1}$ for any $z\in \mathbb{P}^{1}$. Indeed, in this case $A$ is similar to a multiple of a rotation matrix whose corresponding attainable sets are dense in  $\mathbb{P}^{1}$.
Furthermore, Theorem~\ref{lower-bound} allows to determine some sets of matrices for which the existence of $z^{(1)},z^{(2)}$ such that $\tilde{\rho}_{z^{(1)}}(\mathcal{M})<\tilde{\rho}_{z^{(2)}}(\mathcal{M})$, as required in Proposition~\ref{discontinuous}, is satisfied. 
An example of application, which immediately follows from the discussion above and from Remark~\ref{rem-improve}, is given as follows.
\begin{proposition}
\label{particular-case}
Let $n=2$ and assume that there exists $A\in\cup_{k\in \mathbb{N}}\mathcal{M}^k$ with nonreal eigenvalues such that $A^l$ is not proportional to the identity for every nonzero integer $l$ and that there is a single index $\bar h$ satisfying $\min_{h=1,\dots,m}\delta_h=\delta_{\bar h}$, with $A_{\bar h}$ diagonal and not proportional to the identity.
Then $\tilde{\rho}_x(\mathcal{M})$ is discontinuous at each point of $\mathbb{R}^2$.
\end{proposition}
{It is easy to see that the matrices in Example~\ref{ex-urbano} satisfy the assumptions of Proposition~\ref{particular-case}.}
Indeed, the matrix $A_2A_1$ is similar to a rotation matrix of angle $\theta$, with $\theta$ incommensurable with $\pi$ (see~\cite{SICON-feedback} for more details).
{An even simpler application of Proposition~\ref{particular-case} is obtained if one directly replaces the rotation angle $\frac{\pi}4$ in the matrix $A_1$ of Example~\ref{ex-urbano} with any angle incommensurable with $\pi$.}
A further numerical example is given below.
 \begin{example} \label{ex-2densesets}
We consider the set of matrices $\mathcal{M}=\{A_1,A_2,A_3\}$ where
 \[A_1=\begin{pmatrix}-2 &3\\ -6 & 4 \end{pmatrix}\,,\quad A_2=\begin{pmatrix}-0.8 &0\\ 0 & 2 \end{pmatrix}\,,\quad A_3=\begin{pmatrix}2 &-1\\ -2 & -2 \end{pmatrix}.\]
 The assumptions of Proposition~\ref{particular-case} are satisfied since $A_1$ has nonreal eigenvalues which are not proportional to roots of the unit while the minimum singular value is equal to $0.8$ and is associated with the diagonal matrix $A_2$. An application of Theorem~\ref{lower-bound} gives the lower bounds $\tilde{\rho}_- = 1.059$ and $\tilde{\rho}^*_- = 1.0675$  for  $\tilde{\rho}(\mathcal{M})$. As a consequence, there exists a dense subset of $\mathbb{R}^2$ starting from which it is possible to stabilize exponentially the system (with the exponential rate $\tilde{\rho}_x(\mathcal{M})=0.8$) and another dense subset starting from which it is not possible to stabilize the system. 
 \end{example}

\section{Conclusion}
In this paper, we have studied the 
{stabilizability} radius of linear switched systems.  Even though such systems are well known to be extremely complex to analyse, we believe that the 
{stabilizability} radius exhibits particularly complex phenomena (see for instance Example \ref{ex-2densesets}), and on the other hand it has been the topic of very little study in the literature. As an example of this, no method was available in order to provide a nontrivial lower bound on the  
{stabilizability} radius of the matrices in Example \ref{ex-urbano}, even though they were introduced more than 25 years ago.  

{Our lower bounds provide a useful complement to previously available methods, which offer upper bounds (that is, sufficient conditions for stabilizability).  Indeed, when the lower bound is larger than one, one can directly deduce infeasability of the sufficient conditions.}

We have provided two results allowing to improve these lower bounds. In particular, Theorem \ref{lower-bound} provides a lower bound that can be refined by simply iteratively computing longer products of the matrices in the studied set.  We leave open the question of whether this procedure leads to the true value of the 
{stabilizability}  radius (as is the case for other classical algorithms allowing to compute other joint spectral characteristics).  In Section \ref{sec-dependence}, we provide a more theoretical analysis, showing that complex discontinuity phenomena occur, even for quite simple examples.

From a control-theoretic perspective, we believe that the problem studied here is of high importance in the context of formal methods and cyber-physical systems control, where the set of control actions available to the controller is often made of a discrete set. We hope that the present research sheds some light on the complexity of the phenomena at stake, and that it will motivate further research in that direction.

\section*{Acknowledgement}
{ This work was supported by the Engineering and Physical Sciences Research Council, grant EP/N002458/1, by the FNRS, the Innoviris Foundation and the Walloon Region, and by  the iCODE institute, research project of the Idex Paris-Saclay. No new data were created in this study.}

The authors would like to thank J. Ouaknine for organising the Bellairs workshop Algorithmic Aspects of Dynamical Systems, (Barbados, March 2017) where this work was initiated.


\bibliography{biblio}

\end{document}